\documentclass[notitlepage,11pt]{article}
\usepackage{amssymb,amsfonts,amsmath,geometry,esint,comment,upgreek,cases,graphicx} 
\usepackage{scalerel,mathtools,stmaryrd,soul,array}

\catcode`\@=11
\numberwithin{equation}{section}

\catcode`\@=12

\allowdisplaybreaks

\newtheorem{Theorem}{Theorem}[section]
\newtheorem{Lemma}[Theorem]{Lemma}

\newtheorem{Corollary}[Theorem]{Corollary}
\newtheorem{Remark}[Theorem]{Remark}

\def\QED{\hfill {$\square$}\goodbreak \medskip}
\def\proof{\noindent{\textbf{Proof. }}}

\usepackage{todonotes}
\usepackage[T3,T1]{fontenc}
\DeclareSymbolFont{tipa}{T3}{cmr}{m}{n}
\DeclareMathAccent{\invbreve}{\mathalpha}{tipa}{16}

\def\R{\mathbb R}

\def\f{\varphi}

\def\intl{\int\limits}

\def\iintl{\iint\limits}
\def\irn{\int\limits_{\R^n}}
\def\eps{\varepsilon}

\newcommand{\cF}{{\mathcal F}}

\newcommand{\tHs}{H^s_0}
\newcommand{\Hdiesis}{H^s_\#}

\def\Dsn{\left(-\Delta_{n}\right)^{s}\!} 
\def\Dsnhalf{\left(-\Delta_{n}\right)^{\frac{s}{2}}\!}

\def\Dsd{\left(-\Delta_{d}\right)^{s}\!} 
\def\Dshalfd{\left(-\Delta_{d}\right)^{\frac{s}{2}}\!}

\def\DsUno{\left(-\Delta_{k}\right)^{s}\!} 
\def\DshalfUno{\left(-\Delta_{k}\right)^{\frac{s}{2}}\!}

\def\DsBraides{\left(-\Delta_{1+k}\right)^{s}\!} 

\def\Dsnp{\left(-\Delta_{n+k}\right)^{s}\!} 
\def\Dsnphalf{\left(-\Delta_{n+k}\right)^{\frac{s}{2}}\!}

\def\E{{\mathcal E}}
\def\tE{\widetilde {\mathcal E}}

\def\tlambda{\widetilde{\lambda^s_m}}

\begin{document}

\title 
{On the asymptotic behaviour of the restricted Dirichlet  Laplacian of real order $s>0$ on stretching tubes} 

\author{Roberta Musina and Giulio Romani
\footnote{DMIF, Universit\`a di Udine,
via delle Scienze 206, 33100 Udine, Italy. Email: {roberta.musina@uniud.it}, {giulio.romani@uniud.it}.}}
\date{}
\maketitle

\begin{abstract}
We study spectral and variational properties of the (possibly) fractional Dirichlet Laplacian $\Dsnp$, $s>0$, on bounded subsets of $\R^n\times\R^k$ whose extent in one or more directions becomes much larger than in the remaining ones.
We first investigate the asymptotic behaviour of the first eigenvalue on stretching (or thinning) tubes, covering in particular the case of integers $s\ge 2$, for which $\Dsnp$ reduces to a polyharmonic operator. Then we  
	compute the $G$-limit of the rescaled operators, and the $\Gamma$-limit of the corresponding rescaled quadratic forms.

\vskip0.5cm


\end{abstract}

\section{Introduction}
The analysis of differential equations, variational problems, and spectral problems posed on domains whose extent 
in one or more directions becomes much larger than in the remaining ones has attracted considerable attention 
because of its relevance to both theory and applications. 
A typical feature of these problems is the occurrence of a {\it dimension reduction} phenomenon: in the asymptotic regime, 
the limiting behaviour is governed solely by the cross-section of the domain.

Situations of this type naturally arise, for instance, in the study of elastic plates and other thin elastic structures \cite{ABP, Ciarlet, Toupin}, relativity \cite{CMS}, as well as in quantum mechanics \cite{EK_qw}; see also the survey \cite{G}. Remarkable results for second-order equations driven by general uniformly elliptic operators were obtained in \cite{Ch3, ChRou, Y}; see also the monographs \cite{Ch,Ch1} by M. M. Chipot and the references therein. For related results involving nonlocal or higher-order operators, we refer the reader to \cite{AFM, BN, BS, BPS, BPS1, ChCRS, ChowR, ChowR2} and to \cite{BSAN,O}, respectively, and the references therein.

\medskip

We now describe our setting in detail. We consider domains in $\R^{n+k}\equiv  \R^n\times\R^k$
of the form
$$
	\omega^n\times\ell B^k\,, \quad\ \ell\to\infty\,,
$$
where $\omega^n$ is a bounded open set in $\R^n$
, and
$B^k\subset \R^k$ is a bounded and open neighbourhood of the origin.
From the viewpoint of applications, the case of greatest interest is $k=1$,  
$\omega^n\times \ell B^k=\omega^n\times(-\ell,\ell)$.

\medskip

Given any $s>0$, the $s$-power of the Dirichlet Laplacian in $\R^{n+k}$ is formally defined via
$$
\mathcal F_{n+k}[\Dsnp u](\xi,\tau)=\big(|\xi|^2+|\tau|^2\big)^{\!s}\mathcal F_{n+k}[u](\xi,\tau)\,,
$$
where $\mathcal F_{n+k}$ is the unitary Fourier transform of $L^2(\R^{n+k})$ into itself. 

If $s\ge 1$ is an integer, then $\Dsnp$ coincides with the pointwisely defined (poly)harmonic operator, see for instance \cite{GGS} and references therein; if $s\in(0,1)$
the following representation formula holds:
$$
\Dsnp u(z)= C_{n+k,s}~{\rm P.V.}\hskip-0.1cm\int\limits_{\R^{n+k}}
\frac{~u(z)-u(z')}{~|z-z'|^{n+k+2s}}\,dz',
$$
where $C_{n+k,s}=\frac{s2^{2s}}{\Gamma(1-s)}
~\!\frac{\Gamma(\frac{n+k}{2}+s)}{\pi^{\frac{n+k}{2}}}$ and ${\rm P.V.}$ means {\it principal value}.

\medskip

We study the asymptotic behaviour of the so-called {\it restricted} Dirichlet fractional Laplacian 
$$
\Dsnp:~ \tHs(\omega^n\times\ell B^k)\to \tHs(\omega^n\times\ell B^k)',
$$
and of the associated quadratic form
$$
\iintl_{\R^{n+k}}|\Dsnphalf u|^2~\!dxdt=\iintl_{\R^{n+k}}\big(|\xi|^2+|\tau|^2\big)^{\!s}|\cF_{n+k}[u]|^2\,d\xi d\tau\,,\quad u\in \tHs(\omega^n\times\ell B^k)
$$
as $\ell\to\infty$, 
see Section \ref{S:preliminary} for notation and details.

\medskip

\noindent
A dimension reduction phenomenon is described through three complementary perspectives:

$a)$ asymptotic behaviour of the eigenvalues;

$b)$ $G$-convergence of the rescaled operators;

$c)$ $\Gamma$-convergence of the rescaled quadratic forms. 

\medskip

We begin with question $a)$. In the local case  $s=1$, the eigenvalues of the conventional Laplacian $-\Delta_{n+k}$ on $H^1_0(\omega^n\times\ell B^k)$ are readily seen to be
$\lambda^1_\nu(\omega^n)+\ell^{-2}{\lambda^1_j(B^k)}$ for $\nu, j\ge 1$  (see the notation in the next section).
In particular 
\begin{equation}
\label{eq:s=1}
\lambda^1_1(\omega^n\times\ell B^k) = \lambda^1_1(\omega^n)+\frac{\lambda^1_1(B^k)}{\ell^2}\,.
\end{equation}
Notably, this is the only case in which separation of variables is possible.
By contrast, Theorem 1.1 in  \cite{AFM} gives
$$
{\lambda}_1^s(\omega^n)< {\lambda}_1^s(\omega^n\times\ell B^k) < {\lambda}_1^s(\omega^n)+\dfrac{{\lambda}_1^s(B^k)}{\ell^{2s}}\quad \text{if $s\in(0,1)$}
$$
(the convexity assumption on $B^k$  therein is, in fact, not needed for this result).
 
We first extend  \cite[Theorem 1.1]{AFM} to cover the case of higher order eigenvalues. 
\begin{Theorem}
\label{T:eigenvalue1}
Let $s\in(0,1)$. For any integer $m\ge 1$ it holds that
\begin{equation}
\label{eq:weak2}
{\lambda}_1^s(\omega^n) <{\lambda}_m^s(\omega^n\times\ell B^k) < {\lambda}_1^s(\omega^n)+\dfrac{{\lambda}_m^s(B^k)}{\ell^{2s}}\,,
\end{equation}
provided that $\ell>0$ is large enough.
	\end{Theorem}

Theorem \ref{T:eigenvalue1} improves the asymptotic estimate obtained in \cite{ChowR2}, where $\omega^n\times\ell B^k=\omega^n\times (-\ell,\ell)^k$ is required, and an
involved decomposition of the domains is used.

\medskip
As far as we know, little is known about the asymptotic behaviour of 
$\lambda_m^s(\omega^n\times\ell B^k)$ for $s>1$ and $m\ge 1$, including the case where $s\ge 2$ is an integer. For $s=2$ and on the domain 
$(-1,1)\times (\ell,\ell)\subset\R^2$, sharp asymptotic estimates on the first eigenvalue were obtained by Owen in \cite{O}, see also the Appendix.

The next result provides estimates for the first eigenvalue $\lambda_1^s(\omega^n\times\ell B^k)$ as $\ell\to\infty$ in case $s>1$.
Further refinements are provided in the Appendix.

\begin{Theorem}\label{T:eigenvalue2}
	Let $s>1$. For any $\ell>0$ it holds that
	\begin{equation}\label{eq:weak}
		\lambda^s_1(\omega^n)+\frac{\lambda^s_1(B^k)}{\ell^{2s}}< \lambda^s_1(\omega^n\times\ell B^k) < \lambda^s_1(\omega^n)+ sc_s\frac{\lambda^s_1(B^k)}{\ell^{2s}}+
		\gamma_s c_s \frac{1}{\ell^2}\,,
	\end{equation}
	where $c_s=\max\{1,2^{s-2}\}$, and where the constant $\gamma_s>0$ 
	depends only on $s$, $\omega^n$, $B^k$, see (\ref{eq:gamma0}), and satisfies the estimate
	\begin{equation}\label{eq:gamma}
		\gamma_s 
		< \lambda^s_1(\omega^n)^{\frac{s-1}s}\lambda^s_1(B^k)^\frac1s.
	\end{equation}
\end{Theorem}

\medskip 

In the Appendix, we apply Theorems \ref{T:eigenvalue1} and \ref{T:eigenvalue2} to describe the asymptotic behaviour
of the eigenvalues of the Dirichlet Laplacian 
$\DsBraides$ on the {\it thin film}
$$
A_\eps=(-\eps,\eps)\times A\subset \R\times\R^k
$$
as $\eps\to 0^+$, see Corollary \ref{T:Braides1}. This setting is of considerable interest because of its applications. In particular, when $s=2$, it arises in the theory of rigidly clamped Kirchhoff plates; see e.g. \cite{BSAN}. Nonlocal operators on thinning domains  
have attracted significant attention in recent years.
We mention, in particular, the recent papers \cite{BS, BPS, BPS1} and references therein, where dimension reduction phenomena for the {\it regional} fractional Laplacian  
$$
(-\Delta_{A_\eps})^{\!s} u(z)= 
C_{1+k,s}~\! {\rm P.V.}\hskip-0.1cm \int\limits_{A_\eps}\frac{u(z)-u(z')}{|z-z'|^{1+k+2s}}~\!dz'~\!,\quad z\in A_\eps\,~,
\quad s\in(0,1)\,,
$$
have been investigated.

\medskip

We now turn to points $b)$ and $c)$ above. We introduce the self-adjoint operator
\begin{equation}
\label{eq:ellesse}
	\mathcal L^s_\ell= R_\ell^*\circ \Dsnp\circ R_\ell\,,\qquad \mathcal L^s_\ell: \tHs(\omega^n\times B^k)\to\tHs(\omega^n\times B^k)',
\end{equation}
where $R_\ell:\tHs(\omega^n\times B^k)\to \tHs(\omega^n\times\ell B^k)$ is the rescaling operator
\begin{equation}
\label{eq:rescaling}
	R_\ell v(x,t)= \ell^{-\frac{k}{2}}v\big(x,\tfrac{t}{\ell}\big)
\end{equation}
and $R_\ell^*:\tHs(\omega^n\times\ell B^k)'\to \tHs(\omega^n\times B^k)'$ is its adjoint. 
More explicitly, in Section \ref{S:preliminary} we note that
\begin{equation}
\label{eq:Lsell0}
	\langle\mathcal L^s_\ell v,\f\rangle
	= \langle\Dsnp(R_\ell v),R_\ell\f\rangle=
	\iintl_{\R^{n+k}}\big(|\xi|^2+\ell^{-2}|\tau|^2\big)^{\!s}\cF_{n+k}[v]\overline{\cF_{n+k}[\f]}~\!d\xi d\tau
\end{equation}
for $ v,\f\in \tHs(\omega^n\times B^k)$. The corresponding quadratic form is given by
\begin{equation}
\label{eq:Q_ell}
	Q_\ell(v)=\iintl_{\R^{n+k}}|\Dsnphalf (R_\ell v)|^2~\!dxdt=\iintl_{\R^{n+k}}\big(|\xi|^2+\ell^{-2}|\tau|^2\big)^{\!s}|\cF_{n+k}[v]|^2~\!d\xi d\tau\,.
\end{equation}
Our results invoke in a natural way the Hilbert space 
$\Hdiesis(\omega^n\times B^k)$,
which is obtained by completing $\mathcal C^\infty_c(\omega^n\times B^k)$ in $L^2(\R^{n+k})$ with respect to the norm 
\begin{equation}
\label{eq:norm_diesis}
	\|v\|^2_\#=\int\limits_{B^k}dt\irn |\Dsnhalf v|^2~\!dx= \int\limits_{B^k}dt\int\limits_{\R^{n}}|\xi|^{2s}|\mathcal F_n[u(\cdot,t)]|^2~\!d\xi~\!.
\end{equation}
The linear operator
$\Dsn:\Hdiesis(\omega^n\times B^k)\to \Hdiesis(\omega^n\times B^k)'$ is defined via
\begin{equation}
\label{eq:Dsn_diesis}
	\langle \Dsn v,\f\rangle= \intl_{B^k}dt\irn(\Dsnhalf v)(\Dsnhalf\f)~\!dx=
	\int\limits_{B^k}dt \intl_{\R^{n}}|\xi|^{2s}\cF_{n}[v]\overline{\cF_{n}[\f]}~\!d\xi ~\!. 
\end{equation}

The Hilbert space $\Hdiesis(\omega^n\times B^k)$, to which Section \ref{S:space_diesis} is devoted, 
and the following result are central to the remainder of the paper.

\begin{Theorem}\label{Pb_lin_2}
Let $\ell\to \infty$,  $(g_{\ell})_\ell\subset L^2(\omega^n\times B^k)$ be given sequences. Assume that $g_{\ell}\to g$ in $L^2(\omega^n\times B^k)$ for some $g\in L^2(\omega^n\times B^k)$.
For each $\ell$, let  $v_\ell\in \tHs(\omega^n\times B^k)$ be the solution to 
		\begin{equation}\label{eq:Dirichlet_ell}
			\mathcal L^s_\ell v_\ell= g_{\ell}\qquad \text{in }\,\tHs(\omega^n\times B^k)'\,.
		\end{equation}
		Then $v_\ell\to v$ in $\Hdiesis(\omega^n\times B^k)$ as $\ell\to \infty$, where $v$ is the unique solution to 
		\begin{equation}\label{eq:Dirichlet_omega}
			\Dsn v= g\qquad \text{in }\,\Hdiesis(\omega^n\times B^k)'\,.
		\end{equation}
	\end{Theorem}

Theorem \ref{Pb_lin_2} has a couple of remarkable consequences. In case $s\in(0,1)$,
the next result improves \cite[Theorem 1.2]{AFM}, see  Remark \ref{R:AFM1}.

\begin{Theorem}\label{Pb_lin_new}
	Let $\ell\to \infty$,  $f_{\ell}=f_{\ell}(x,t)\in L^2(\omega^n\times \ell B^k)$ be given sequences.  
	Assume that $f_{\ell}\to f_\infty$ in $L^2-$mean  for some $f_\infty=f_\infty(x)\in L^2(\omega^n)$, that is,
	$$
	\fint\limits_{\ell B^k}dt\int\limits_{\omega^n} |f_{\ell}-f_\infty|^2~\!dx\to 0\,.
	$$
	For each $\ell$, let $u_{\ell}=u_{\ell}(x,t)\in \tHs(\omega^n\times \ell B^k)$ be the solution to
	\begin{equation}\label{eq:Dirichlet_ell_Thm}
	\Dsnp u_{\ell}= f_{\ell}\quad\text{in $\tHs(\omega^n\times \ell B^k)'$}~\!,
	\end{equation}
	and let $u_\infty=u_\infty(x)\in \tHs(\omega^n)$ be the solution to
	\begin{equation}
	\label{eq:Dirichlet_omega_Thm}
	\Dsn u_\infty= f_\infty\quad\text{in $\tHs(\omega^n)'$}\,.
	\end{equation}
	Then $\Dsnhalf u_{\ell}\to \Dsnhalf u_\infty$ in $L^2-$mean, that is,
	\begin{equation}
	\label{eq:tesi}
	\fint\limits_{\ell B^k}dt\irn\big|\Dsnhalf u_{\ell}-\Dsnhalf u_\infty\big|^2 dx \to 0\,.
	\end{equation}
\end{Theorem}

 Roughly speaking, we see that while the cross-section $\omega^n$ becomes somehow negligible for $\ell$ large, the problems (\ref{eq:Dirichlet_ell_Thm})
on $\omega^n\times \ell B^k$ converge to the problem (\ref{eq:Dirichlet_omega_Thm}), which is settled on $\omega^n$.

\medskip

Next, we view $\mathcal L^s_\ell$ as an unbounded operator $L^2(\R^{n+k})\to L^2(\R^{n+k})$,
with domain 
$$
D(\mathcal L^s_\ell)=\big\{v\in \tHs(\omega^n\times B^k)~|~ \Dsnp v\in L^2(\omega^n\times B^k)~\!\big\}\,
$$
(meaning that, for $v\in D(\mathcal L^s_\ell)$ the restrictions of $\Dsnp v$ and $\mathcal L^s_\ell v$ to $\omega^n\times B^k$ belong to $L^2(\omega^n\times B^k)$).

By taking $g_\ell=g$ independent of $\ell$ in Theorem \ref{Pb_lin_2}, in view of \cite[Definition 13.3]{DM} we 
immediately infer the next result.

\begin{Corollary}
\label{C:G-convergence}
Let $\ell\to\infty$ be any divergent sequence. 
Then  the sequence of unbounded operators $\mathcal L^s_\ell:L^2(\R^{n+k})\to L^2(\R^{n+k})$  $G$-converges in the strong $L^2$-topology to the operator
$$
\Dsn :L^2(\R^{n+k})\to L^2(\R^{n+k})~\!,
$$
compare with (\ref{eq:Dsn_diesis}), with domain
$$
D(\Dsn\,)=\bigl\{u\in\Hdiesis(\omega^n\times B^k)~|~\Dsn u\in L^2(\R^n\times B^k)\bigr\}~\!.
$$
\end{Corollary}

Finally, we extend the quadratic form $Q_\ell: \tHs(\omega^n\times B^k)\to\R$, defined in (\ref{eq:Q_ell}),
to $L^2(\R^{n+k})$ by setting $Q_\ell\equiv \infty$ outside $\tHs(\omega^n\times B^k)$.
Then, Corollary \ref{C:G-convergence}, together with \cite[Theorem 13.5]{DM}, immediately yields
the following consequence.

\begin{Corollary}\label{Thm-Gamma}
	Let $\ell\to\infty$ be any divergent sequence. Then
	$$
	\Gamma\!-\!\displaystyle\lim_{\ell\to\infty}Q_\ell (v)=
	\begin{cases}
		\displaystyle\int\limits_{B^k}dt\int\limits_{\R^n}|\Dsnhalf v|^2dx&\text{if}\ \,v\in\Hdiesis(\omega^n\times B^k)\,,\\
		\,\infty&\text{elsewhere in}\ \,L^2(\R^{n+k})\,,
		\end{cases}
	$$
	in the strong $L^2$-topology.
\end{Corollary}

In case $s\in(0,1)$ and $B^k$ convex, the $\Gamma$-limit of the quadratic forms $Q_\ell$ has been computed only for
functions $v$ belonging to the smaller space $\tHs(\omega^n\times B^k)$, compare with  \cite[Theorem 4.2]{AFM}.

\medskip
The paper is organized as follows. 
In Section \ref{S:preliminary} we introduce the notation, recall few known results, 
and establish formula \eqref{eq:Lsell0}. 
Theorems \ref{T:eigenvalue1} and \ref{T:eigenvalue2} are proved in Section \ref{S:eigenvalues}; Section \ref{S:space_diesis} is devoted to the study of the space ${\Hdiesis(\omega^n\times B^k)}$, and Section \ref{S:GGamma} contains the proofs of Theorems \ref{Pb_lin_2} and \ref{Pb_lin_new}. Finally, in the Appendix we state and prove Corollary \ref{T:Braides1}, concerning (nonlocal) problems on thinning tubes, and present some remarks and refinements of our main results.

\vskip0.2truecm

\section{Notation and preliminaries}
\label{S:preliminary}

Our main reference for fractional Sobolev spaces is H. Triebel's monograph \cite{Tr}; 
for $G$- and $\Gamma$-convergence, we refer to the monographs by G. Dal Maso \cite{DM} and A. Braides \cite{Braides}.

\paragraph{Notation.} Let $d\ge 1$ be an integer and
let $D\subset \R^d$ be open (the cases of interest are $d=n$, $D=\omega^n$,  and $d=n+k$, $D=\omega^n\times\ell B^k$).
Any   $f\in L^2(D)$ is extended by the null function
outside $D$, so that  we have the isometry
\begin{equation}
\label{eq:null}
L^2(D)\hookrightarrow L^2(\R^d)~\!.
\end{equation}
In a similary way, we agree that $\mathcal C^\infty_c(D)$ is a subspace of $\mathcal C^\infty_c(\R^d)$.

For $u\in L^2(\R^d)$ the Fourier transform $\cF_d[u]$ is defined by
$$
	\mathcal F_d[u](\zeta)=\frac1{(2\pi)^{d/2}}\intl_{\R^d} e^{-iz\cdot\zeta}u(z)~\!dz\,.
$$
Given any real number $s>0$, the Sobolev space
$$
H^s(\R^d) = \big\{u\in L^2(\R^d)~|~ \Dshalfd u:=\mathcal F_d^{-1}\big[|\,\cdot\,|^{2s}\mathcal F_d[u]\big]\in L^2(\R^d)~\!\big\}
$$
naturally inherits a Hilbertian structure with norm
$$
\|u\|^2 = \intl_{\R^d}|\Dshalfd u|^2~\!dz+ \intl_{\R^d}|u|^2~\!dz= \intl_{\R^d}(|\zeta|^{2s}+1)|\cF_d [u]|^2~\!d\zeta~\!.
$$
For  a given (possibly unbounded) open set $D\subset\R^d$  we put
$$
\tHs(D)=\overline{\mathcal C^\infty_c(D)}^{H^s(\R^{d})}
$$
(denoted by $\mathring{H}^{s}$ in Triebel's monograph \cite{Tr}),
which is a closed subspace of $H^s(\R^d)$. 
If $D$ is bounded, then $\tHs(D)$ is compactly embedded into $L^2(D)$, thus the spectrum of $\Dsd$ on $\tHs(D)$ is
discrete and consists of a non decreasing, divergent sequence of eigenvalues
$$
0<{\lambda^s_1}(D)\le {\lambda^s_2}(D)\le {\lambda^s_3}(D)\le\dots\,,
$$
where
$$
\lambda_1^s(D)=\inf_{v\in \tHs(D)\atop v\neq 0}
\frac{\displaystyle\intl_{\R^{d}}|\Dshalfd v|^2~\!dz}{\displaystyle\intl_{D}|v|^2\,dz}~\!.
$$
Each eigenvalue is repeated according to its multiplicity. 
In particular, the first eigenvalue $\lambda^s_1(D)$ may have multiplicity greater than one when $s>1$,
see for instance \cite[Chapter 3]{GGS} for $s=2$.
On the other hand, if $s\in(0,1]$ then $\lambda_1^s(D)$ is simple, hence ${\lambda^s_1}(D)< {\lambda^s_2}(D)$.

We identify $L^2(D)$ with a subspace of $\tHs(D)'$, which is the dual of $\tHs(D)$.
Given $f\in L^2(D)$, we say that $u\in \tHs(D)$ is a  solution to 
$$
\Dsd u=f\qquad \text{in $\tHs(D)'$}
$$
if $u$ solves $\Dsd u=f$ is a weak sense, that is, if 
$$\langle \Dsd u,\f\rangle=\intl_{\R^d}(\Dshalfd u)~\!{(\Dshalfd \f)}~\!dz=\intl_{\R^d} f\f~\!dz\qquad \text{for any }\f\in \tHs(D)\,.$$
A similar notation is used for equations involving $\mathcal L_\ell^s:\tHs(\omega^n\times B^k)\to \tHs(\omega^n\times B^k)'$, see \eqref{eq:ellesse}.

\paragraph{Preliminaries.} We recall few simple facts taken from \cite{AFM}. 
For $u\in \tHs(\omega^n\times\ell B^k)$ we set
$$
\widetilde\E^s_{\ell}(u)= \intl_{\ell B^k}dt\irn|\Dsnhalf u|^2~\!dx+\intl_{\omega^n}dx\intl_{\R^k}|\DshalfUno u|^2\,dt\,.
$$
In the local case $s=1$ we trivially have
$$
\widetilde\E^{\,1}_{\ell}(u)=\intl_{\ell B^k}dt\intl_{\omega^n}|\nabla_{\!x} u|^2~\!dx+\intl_{\omega^n}dx\intl_{\ell B^k}|\nabla_{\!t} u|^2~\!dt
=\iint\limits_{\omega^n\times\ell B^k}|\left(-\Delta_{n+k}\right)^{\frac12}\! u|^2\,dxdt\,.
$$
In contrast, in the (possibly) nonlocal case $s\neq1$ it was shown in \cite[Lemma 2.2 and Remark 2.3]{AFM}
that this nice splitting never occurs. In fact, the next lemma holds.

\begin{Lemma}
\label{L:poincare}
Let $s\neq 1$ and let $u\in\tHs(\omega^n\times\ell B^k)$ be nontrivial. Then $u(\cdot,t)\in \tHs(\omega^n)$ for a.e. $t\in\R^k$ and $u(x,\cdot)\in H^s_0(B^k)$ for a.e. $x\in \R^n$. In addition, for any $\ell>0$, 
\begin{gather}
	\label{eq:first}
	{\lambda^s_1}(\omega^n)\iintl_{\omega^n\times\ell B^k}|u|^2 dxdt< \intl_{\ell B^k}dt\irn |\Dsnhalf u|^2~\!dx< \iint\limits_{\R^{n+k}} |\Dsnphalf u|^2~\! dxdt\,,\\
	\label{eq:second_bis}
	\min\{1,2^{s-1}\} \widetilde\E^s_{\ell}(u) <\iint\limits_{\R^{n+k}} |\Dsnphalf u|^2~\! dxdt<\max\{1,2^{s-1}\}\widetilde\E^s_{\ell}(u)\,.
\end{gather}
\end{Lemma}

Thanks to Lemma \ref{L:poincare}, we can introduce the operator  $\Dsn\oplus\DsUno: \tHs(\omega^n\times\ell B^k)\to \tHs(\omega^n\times\ell B^k)'$ defined by
$$
\begin{aligned}
\langle [\Dsn&\oplus\DsUno~\!]u,v\rangle\\
&= \intl_{\ell B^k}dt\irn(\Dsnhalf u)(\Dsnhalf v)~\!dx+\intl_{\omega^n}dx\intl_{\R^k}(\DshalfUno u)(\DshalfUno v)\,dt\,
\end{aligned}
$$
for $u,v\in \tHs(\omega^n\times\ell B^k)$. Notice that
$$
\begin{aligned}
\langle [\Dsn\oplus\DsUno~\!]u,v\rangle
&=\intl_{\ell B^k}dt\irn|\xi|^{2s}\mathcal F_n[u]\overline{\mathcal F_n[v]}~\!d\xi+\intl_{\omega^n}dx\intl_{\R^k}|\tau|^{2s}\mathcal F_k[u]\overline{\mathcal F_k[v]}~\!d\tau\,.
\end{aligned}
$$
By (\ref{eq:second_bis}), the quadratic form $(u,v)\mapsto \langle [\Dsn\oplus\DsUno~\!]u,v\rangle$
provides an equivalent Hilbertian scalar product in $\tHs(\omega^n\times\ell B^k)$, with corresponding norm given by
$\|u\|^2=\tE_\ell^s(u)$. It follows that the spectrum of $\Dsn\oplus\DsUno$ on $\tHs(D)$ is
discrete and consists of a divergent sequence of positive eigenvalues, which will be computed in the following lemma.

\begin{Lemma}
\label{L:tilde}
The eigenvalues of $\Dsn\oplus\DsUno$ on $\tHs(\omega^n\times\ell B^k)$ have the form
$$
\lambda^s_\nu(\omega^n)+\frac{\lambda^s_j(B^k)}{\ell^{2s}}\,,\quad \nu, j\ge 1.
$$
\end{Lemma}

\proof
The proof is obtained via separation of variables and elementary computations. We add few details for completeness.
We take an orthonormal basis $\{\f_\nu\}_\nu\subset\tHs(\omega^n)$ of $L^2(\omega^n)$ consisting by solutions  to 
$$
	\Dsn \f_\nu={\lambda^s_\nu}(\omega^n)\f_\nu\qquad \text{in $\tHs(\omega^n)'$\,.}
$$
These functions satisfy
$$
	\irn|\xi|^{2s}\cF_n[\f_\nu]~\!\overline{\cF_n[\f_h]}~\!d\xi={\lambda^s_\nu}(\omega^n)\int\limits_{\omega^n}\f_\nu\f_h dx={\lambda^s_\nu}(\omega^n)\delta_{\nu h}\,.
$$	
Then we take an orthonormal basis of $L^2(B^k)$ consisting by eigenfunctions $\psi_j\in \tHs(B^k)$ of $\DsUno$ on $\tHs(B^k)$, and put
$\psi_j^\ell(t)=\ell^{-\frac{k}{2}}\psi_j(\ell^{-1}{t})$, so that $\mathcal F_k[\psi^\ell_j](\tau)=\ell^\frac{k}{2}\mathcal F_k[\psi](\ell\tau)$.
It is easy to check that the family $\{\psi_j^\ell\}_j$ is an orthonormal basis for $L^2(\ell B^k)$ and moreover
$$
\begin{gathered}
\DsUno \psi_j^\ell=\frac{\lambda^s_j(B^k)}{\ell^{2s}}\psi^\ell_j\qquad \text{in $\tHs(B_\ell^k)'$,}\\
	\intl_{\R^k}|\tau|^{2s}\cF_k[\psi^\ell_j]~\!\overline{\cF_k[\psi^\ell_i]}~\!d\tau=\frac{\lambda^s_j(B^k)}{\ell^{2s}}\int\limits_{\ell B^k}\psi^\ell_j\psi^\ell_i dt=\frac{\lambda^s_j(B^k)}{\ell^{2s}}\delta_{ji}\,.
	\end{gathered}
$$	
Next, we observe that for each pair of indexes $\nu, j\ge 1$, the function
$$
\Phi_{\nu j}^\ell(x,t):=
\f_\nu(x)\psi^\ell_j(t)
$$
belongs to $\tHs(\omega^n\times\ell B^k)$ and solves 
$$
[\Dsn\oplus\DsUno~\!]\Phi_{\nu j}^\ell=\big(\lambda^s_\nu(\omega^n)+\ell^{-2s}\lambda^s_{j}(B^k)\big)\Phi_{\nu j}^\ell\qquad \text{in $\tHs(\omega^n\times\ell B^k)'$\,.}
$$
Moreover, the family 
$\{\Phi_{\nu j}^\ell\}_{\nu,j\ge 1}$ provides an orthonormal basis for $L^2(\omega^n\times\ell B^k)$ and an orthogonal basis for $\tHs(\omega^n\times\ell B^k)$. 
This means that for any $u\in \tHs(\omega^n\times\ell B^k)$, it holds that 
$$
u=\sum_{\nu,j=1}^\infty u_{\nu j}
\Phi_{\nu j}^\ell 
~,\quad \text{where}\quad u_{\nu j}=\iintl_{\omega^n\times\ell B^k} u(x,t)\f_\nu(x)\psi^\ell_j(t)~\!dxdt\,,
$$
with convergence in $L^2(\omega^n\times\ell B^k)$, and
$$
[\Dsn\oplus\DsUno~\!]u= \sum_{\nu,j=1}^\infty \big(\lambda_\nu^s(\omega^n)+\ell^{-2s}\lambda_j^s(B^k)\big)u_{\nu j}
\Phi_{\nu j}^\ell\,,
$$
with convergence in the dual space $\tHs(\omega^n\times\ell B^k)'$. 
It follows that if $u\in \tHs(\omega^n\times\ell B^k)\setminus\{0\}$ solves 
$$
[\Dsn\oplus\DsUno~\!]u=\lambda u\qquad \text{in $\tHs(\omega^n\times\ell B^k)'$}
$$
for some $\lambda\in\R$,
then
$$
\lambda u_{\nu j} \Phi_{\nu j}^\ell= 
\big(\lambda_\nu^s(\omega^n)+\ell^{-2s}\lambda_j^s(B^k)\big)u_{\nu j}\Phi_{\nu j}^\ell
$$
for any pair of indexes $\nu, j\ge 1$. Thus there exists a unique pair of indexes $\nu, j$ such that $u_{\nu j}\neq 0$ and
 $\lambda= \lambda_\nu^s(\omega^n)+\ell^{-2s}\lambda_j^s(B^k)$. 
\QED 

We conclude this preliminary section by proving formula \eqref{eq:Lsell0}, which follows by a straightforward computation.

\paragraph{Proof of (\ref{eq:Lsell0})}
Notice that  
$\mathcal F_{n+k}[R_\ell w](\xi,\tau)= \ell^\frac{k}{2} \mathcal F_{n+k}[w](\xi,\ell\tau)$ for  $w\in \mathcal C^\infty_c(\omega^n\times B^k)$. Thus, 
for any $v,\f\in \mathcal C^\infty_c(\omega^n\times B^k)$ we can compute
$$
\begin{aligned}
\langle \mathcal L^s_\ell v,\f\rangle &
=\iintl_{\R^{n+k}}\big(|\xi|^2+|\tau|^2\big)^{\!s}\mathcal F_{n+k}[R_\ell v]\overline{\mathcal F_{n+k}[R_\ell \f]}~\!d\xi d\tau
\\
&=\ell^{k}\iintl_{\R^{n+k}}\big(|\xi|^2+|\tau|^2\big)^{\!s}\mathcal F_{n+k}[v](\xi,\ell\tau)\overline{\mathcal F_{n+k}[\f](\xi,\ell\tau)}~\!d\xi d\tau
\\
&=\iintl_{\R^{n+k}}\big(|\xi|^2+\ell^{-2}|\tau|^2\big)^{\!s}\mathcal F_{n+k}[v](\xi,\tau)\overline{\mathcal F_{n+k}[\f](\xi,\tau)}~\!d\xi d\tau~\!.
\end{aligned}
$$
Since $\mathcal C^\infty_c(\omega^n\times B^k)$ is dense in $\tHs(\omega^n\times B^k)$, the proof is complete. 
\QED

\section{Eigenvalues}
\label{S:eigenvalues}

\paragraph{Proof of Theorem \ref{T:eigenvalue1}}
Since $s\in(0,1)$, then (\ref{eq:second_bis}) gives
$$
\iint\limits_{\R^{n+k}} |\Dsnphalf u|^2~\! dxdt\le \widetilde\E^s_{\ell}(u) \quad\ \, \text{for any $u\in\tHs(\omega^n\times\ell B^k)$\,.}
$$
Therefore, the Courant–Fischer–Weyl variational principle implies that
$\lambda^s_m(\omega^n\times\ell B^k)\le \tlambda(\omega^n\times\ell B^k)$ for any fixed $m\ge 1$, where $\tlambda(\omega^n\times\ell B^k)$ is the $m$-th eigenvalue of the operator $\Dsn\oplus\DsUno$ on $\tHs(\omega^n\times\ell B^k)$. Since these eigenvalues are both achieved, it holds that
$$
\lambda^s_m(\omega^n\times\ell B^k)< \tlambda(\omega^n\times\ell B^k)\,.
$$

\medskip

From Lemma \ref{L:tilde} we
infer that (\ref{eq:weak2}) holds, provided that 
$\lambda^s_1(\omega^n)+\ell^{-2s}{\lambda^s_m(B^k)}$ coincides with the $m$-th eigenvalue of $\Dsn\oplus\DsUno$ on $\tHs(\omega^n\times\ell B^k)$,
that is, if 
$$
\lambda^s_1(\omega^n)+\frac{\lambda^s_m(B^k)}{\ell^{2s}}
\le \lambda^s_2(\omega^n)+\frac{\lambda^s_1(B^k)}{\ell^{2s}}\,.
$$
Since $s\in(0,1)$, we know that $\lambda^s_1(\omega^n)$ is simple. Thus (\ref{eq:weak2}) holds true, provided that 
$\ell>0$ is large enough.
The proof is complete.
 \QED

\paragraph{Proof of Theorem \ref{T:eigenvalue2}.}
Recall that 
$$
{\lambda^s_1}(\omega^n)=\inf_{v\in \tHs(\omega^n)\atop v\neq 0}
\frac{\displaystyle\intl_{\R^n}|\Dsnhalf v|^2~\!dx}{\displaystyle\intl_{\omega^n}|v|^2\,dx}
\,,\quad
{\lambda^s_1}(\omega^n\times\ell B^k)=\inf_{v\in \tHs(\omega^n\times\ell B^k)\atop v\neq 0}
\frac{\displaystyle\iintl_{\R^{n+k}}|\Dsnphalf v|^2~\!dxdt}{\displaystyle\iintl_{\omega^n\times\ell B^k}|v|^2\,dxdt}\,.
$$
The first eigenvalue $\lambda_1^s(\omega^n\times\ell B^k)$ is achieved by a  function 
$\phi_\ell\in \tHs(\omega^n\times\ell B^k)\setminus\{0\}$. Since $s>1$, then (\ref{eq:second_bis}) in Lemma \ref{L:poincare} gives
$$
\begin{aligned}
	\iint\limits_{\R^{n+k}} |\Dsnphalf \phi_\ell|^2~\! dxdt 
	&> \intl_{\ell B^k}dt\irn|\Dsnhalf \phi_\ell|^2~\!dx+\intl_{\omega^n}dx\intl_{\R^k}|\DshalfUno \phi_\ell|^2\,dt\\
	&\ge\lambda_1^s(\omega^n)\intl_{\ell B^k}dt\intl_{\omega^n}|\phi_\ell|^2~\!dx+\lambda_1^s(\ell B^k)\intl_{\omega^n}dx\intl_{\ell B^k}|\phi_\ell|^2\,dt\,.
\end{aligned}
$$
It is readily proved that $\lambda_1^s(\ell B^k)=\ell^{-2s}\lambda_1^s(B^k)$, use the rescaling function $R_\ell$ in (\ref{eq:rescaling}). This leads to
$$
	\lambda_1^s(\omega^n\times\ell B^k)\!\iintl_{\omega^n\times\ell B^k}\!|\phi_\ell|^2~\! dxdt= \iint\limits_{\R^{n+k}} |\Dsnphalf \phi_\ell|^2~\! dxdt 
	>
	 \big(\lambda_1^s(\omega^n)+ \ell^{-2s} \lambda_1^s(B^k)\big)\!\iintl_{\omega^n\times\ell B^k}\!|\phi_\ell|^2~\! dxdt\,,
$$
which concludes the proof of 
the left-hand side inequality in (\ref{eq:weak}).

\medskip

To estimate $\lambda^s_1(\omega^n\times\ell B^k)$ from above we take $\f_1\in \tHs(\omega^n)$ solving 
$$
\Dsn \f_1={\lambda^s_1}(\omega^n)\f_1\qquad \text{in $\tHs(\omega^n)'$.}
$$
We normalize $\f_1$ in $L^2(\omega^n)$, so that
\begin{equation}
	\label{eq:fi1}
	\irn|\xi|^{2s}|\cF_n[\f_1]|^2~\!d\xi=\irn|\Dsnhalf \f_1|^2~\!dx={\lambda^s_1}(\omega^n)\int\limits_{\omega^n}|\f_1|^2dx={\lambda^s_1}(\omega^n)\,.
\end{equation}
Next, we take an eigenfunction $\psi_1\in \tHs(B^k)$ associated with the eigenvalue $\lambda^s_1(B^k)$.
After normalization,
this function satisfies
\begin{equation}
	\label{eq:psij1}
	\intl_{\R^k}|\tau|^{2s}|\cF_k[\psi_1]|^2~\!d\tau=\intl_{\R^k}|\DshalfUno \psi_1|^2~\!dt={\lambda^s_1}(B^k)\int\limits_{B^k}|\psi_1|^2dt={\lambda^s_1}(B^k)\,.
\end{equation}
Our aim is to prove that (\ref{eq:weak}) holds with
\begin{equation}
\label{eq:gamma0}
\gamma_s =\irn|\left(-\Delta_{n}\right)^{\frac{s-1}{2}}\f_1|^2dx~\intl_{\R^k}|\nabla_{\!t}\psi_1|^2dt\,.
\end{equation}
To this end we set
$$
\psi_1^\ell(t)=\ell^{-\frac{k}{2}}\psi_1\Big(\frac{t}{\ell}\Big)\,.
$$
Since $\f_1\psi_1^\ell\in \tHs(\omega^n\times\ell B^k)$ has unitary $L^2$-norm, we obtain
\begin{equation}
	\label{eq:max}
	\lambda_1^s(\omega^n\times\ell B^k)\le 
	{\displaystyle\iintl_{\R^{n+k}}|\Dsnphalf (\f_1\psi_1^\ell)|^2~\!dxdt}~\!.
\end{equation}
Now, we plainly get
$$
\iint\limits_{\R^{n+k}}|\Dsnphalf (\f_1\psi_1^\ell)|^2~\!dxdt=
\iintl_{\R^{n+k}}\big(|\xi|^2+\ell^{-2}|\tau|^2\big)^{\!s}|\cF_{n+k}[\f_1\psi_1]|^2d\xi d\tau\,.
$$
Using (\ref{eq:fi1}) and (\ref{eq:psij1}) we can compute
\begin{equation}
	\label{eq:start}
	\begin{aligned}
		\int\limits_{\R^n} |\xi|^{2s}|\mathcal F_n[\f_1]|^2\,d\xi\int\limits_{\R^k}|\cF_k[\psi_1]|^2\,d\tau&=\lambda_1^s(\omega^n)\,,\\
		\int\limits_{\R^n} |\mathcal F_n[\f_1]|^2\,d\xi\int\limits_{\R^k}|\tau|^{2s}|\cF_k[\psi_1]|^2\,d\tau&= \lambda_1^s(B^k)\,.
	\end{aligned}
\end{equation}
Since $s>1$,  
the Lagrange mean value theorem gives $\big(1+\zeta\big)^{\!s}< 1+s\big(1+\zeta\big)^{\!s-1}\zeta$ for any $\zeta>0$. Thus
\begin{equation}
\label{eq:elementary}
\big(1+\zeta\big)^{\!s}< 
1+sc_s\big(\zeta+\zeta^s\big)\qquad \text{for any $\zeta>0$\,,}
\end{equation}
where $c_s=\max\{1,2^{s-2}\}$. By setting $\zeta=\ell^{-2}|\tau|^2|\xi|^{-2}$, we infer that
$$
	\big(|\xi|^2+\ell^{-2}|\tau|^2\big)^{\!s}<|\xi|^{2s}+sc_s\ell^{-2s}|\tau|^{2s}+sc_s\ell^{-2}|\xi|^{2(s-1)}|\tau|^2\qquad \text{for a.e. $\xi\in\R^n, \tau\in\R^k$.}
$$
Taking also 
(\ref{eq:max}) and (\ref{eq:start}) into account, we can therefore estimate
$$
\begin{aligned}
\lambda_1^s(\omega^n\times\ell B^k)&<\lambda_1^s(\omega^n)+sc_s\ell^{-2s}\lambda_1^s(B^k)
+sc_s\ell^{-2}\int\limits_{\R^n} |\xi|^{2(s-1)}|\mathcal F_n[\f_1]|^2\,d\xi\int\limits_{\R^k}|\tau|^2|\cF_k[\psi_1]|^2\,d\tau
\\
&=\lambda_1^s(\omega^n)+sc_s\ell^{-2s}\lambda_1^s(B^k)
+sc_s\ell^{-2}\Big(\irn|\left(-\Delta_{n}\right)^{\frac{s-1}{2}}\f_1|^2dx\Big)\Big(\intl_{\R^k}|\nabla_{\!t}\psi_1|^2dx\Big)\,.
\end{aligned}
$$
It remains to prove the inequality in (\ref{eq:gamma}). Let $s'$ be the conjugate of $s$. We use the H\"older inequality twice, (\ref{eq:fi1}) and \eqref{eq:psij1} to get 
$$
\begin{aligned}
	\int\limits_{\R^n} |\xi|^{2(s-1)}|\mathcal F_n[\f_1]|^2\,d\xi&< \Big(\int\limits_{\R^n} |\xi|^{2s}|\mathcal F_n[\f_1]|^2\,d\xi\Big)^\frac{1}{s'}
	\Big(\int\limits_{\R^n} |\mathcal F_n[\f_1]|^2\,d\xi\Big)^\frac{1}{s}=\lambda_1^s(\omega^n)^\frac{1}{s'}\,,
	\\
	\int\limits_{\R^k}|\tau|^2|\cF_k[\psi_1]|^2\,d\tau&<
	\Big(\int\limits_{\R^k}|\tau|^{2s}|\cF_k[\psi_1]|^2\,d\tau\Big)^\frac{1}{s}
	\Big(\int\limits_{\R^k}|\cF_k[\psi_1]|^2\,d\tau\Big)^\frac{1}{s'}= \lambda_1^s(B^k)^\frac{1}{s}\,,
\end{aligned}
$$
which completes the proof.
\QED

\section{The space ${{\Hdiesis(\omega^n\times B^k)}}$}
\label{S:space_diesis}

Thanks to Lemma \ref{L:poincare}, we can introduce the space
	$\Hdiesis(\omega^n\times B^k)$ obtained by completing $\mathcal C^\infty_c(\omega^n\times B^k)$ with respect to the norm in (\ref{eq:norm_diesis}).
	It turns out that $\tHs(\omega^n\times B^k)$ is a dense subspace of $\Hdiesis(\omega^n\times B^k)$, and that $\Hdiesis(\omega^n\times B^k)\hookrightarrow L^2(\omega^n\times B^k)$ with continuous embedding. In fact, we have 
	$$
			\intl_{B^k}dt\irn |\Dsnhalf u|^2~\!dx\ge {\lambda^s_1}(\omega^n)\iintl_{\omega^n\times B^k}|u|^2 dxdt\quad \text{for any }\,u\in \Hdiesis(\omega^n\times B^k)\,.
$$	
		The embedding $\Hdiesis(\omega^n\times B^k)\hookrightarrow L^2(\omega^n\times B^k)$ is evidently not compact. For instance, let
		$(\eta_h)_h\subset \mathcal C^\infty_c(B^k)$ be a sequence such that $\|\eta_h\|_{L^2(B^k)}=1$ and $\eta_h\to 0$ weakly in $L^2(B^k)$; then take 
		any $v\in \mathcal C^\infty_c(\omega^n)\setminus\{0\}$. 
		Then the sequence $(\eta_hv)_h$ is bounded in $\Hdiesis(\omega^n\times B^k)$ because
		$$
		\|\eta_hv\|_\#^2=\int\limits_{B^k}\eta_h^2\,dt\,\irn|\Dsnhalf v|^2dx=\irn|\Dsnhalf v|^2dx\,,
$$
and it converges weakly but not strongly to $0$ in $L^2(\omega^n\times B^k)$.

\medskip

By standard arguments, one can prove that $\Hdiesis(\omega^n\times B^k)$ is a Hilbert space with scalar product
			$$
			(v,w)_\# = \intl_{B^k}dt\intl_{\R^n}(\Dsnhalf v)(\Dsnhalf w)~\!dx~\!.
			$$
The differential of the quadratic form 
$$
	v\mapsto \frac12\intl_{B^k}dt\irn |\Dsnhalf v|^2dx=\frac12\intl_{B^k}dt\irn|\xi|^{2s}|\cF_n[v(\cdot,t)]|^2d\xi\,,\quad v\in\Hdiesis(\omega^n\times B^k)
	$$
	is denoted by $\Dsn: \Hdiesis(\omega^n\times B^k)\to \Hdiesis(\omega^n\times B^k)'$, compare with (\ref{eq:Dsn_diesis}).

	We now provide a more concrete description of the space $\Hdiesis(\omega^n\times B^k)$.

	\begin{Lemma}\label{embeddings}
		Let $u\in L^2(\R^n\times B^k)\subset  L^2(\R^{n+k})$ (compare with (\ref{eq:null})). Then $u\in\Hdiesis(\omega^n\times B^k)$ if and only if
			\begin{equation}\label{eq:space}
				u(\cdot,t)\in \tHs(\omega^n)~~\text{for a.e. $t\in B^k$,\quad and} \quad\intl_{B^k}dt \intl_{\R^n} |\Dsnhalf u|^2~\!dx<\infty\,.
			\end{equation}
	\end{Lemma}

	\proof
		The "only if" part is immediate from the definition of $\Hdiesis(\omega^n\times B^k)$. 
		
		Conversely, let $u\in L^2(\R^n\times B^k)$ satisfy (\ref{eq:space}) and  let  $\eps>0$ be arbitrarily chosen.
		Recall that $u$ is extended by the null function outside $\R^n\times B^k$ by (\ref{eq:null}).
		Take a sequence $\eta_\nu=\eta_\nu(t)\in \mathcal C^\infty_c(B^k)$ such that $0\le \eta_\nu\le 1$ and such that $\eta_\nu\to 1$ a.e. on $B^k$ 
		as $\nu\to\infty$. 
		By Lebesgue's convergence theorem, there exists $\nu_\eps>0$ large enough, such that 
		\begin{gather}
			\nonumber
			\intl_{B^k}dt \int\limits_{\omega^n}|u-\eta_{\nu_\eps}u|^2~\!dx=\intl_{B^k}dt\int\limits_{\omega^n}|1-\eta_{\nu_\eps}|^2|u|^2~\!dx<\eps\,,\\
			\label{eq:A}
			\intl_{B^k}dt \intl_{\R^n}|\Dsnhalf(u-\eta_{\nu_\eps}u)|^2dx = \intl_{B^k}dt \intl_{\R^n}|1-\eta_{\nu_\eps}|^2|\Dsnhalf u|^2dx<\eps~\!.
		\end{gather}
		Next, we take a sequence $\rho_h=\rho_h(t)$ of mollifiers in $\mathcal C^\infty_c(B^k)$ and we denote by $*_k$ the convolution in $B^k$. Explicitly, we have
		$$
		\big( (\eta_{\nu_\eps}u)*_k\rho_h\big)(x,t)=\int\limits_{B^k}\eta_{\nu_\eps}(\tau)u(x,\tau)\rho_h(t-\tau)~\!d\tau\,.
		$$
		The function $(x,t)\mapsto \big( (\eta_{\nu_\eps}u)*_k\rho_h\big)(x,t)$ belongs to $L^2(\R^{n+k})$, is smooth in the $t$-variable
		for a.e. $x\in\R^n$ and, 
		for $h$ sufficiently large, has compact support in $\overline{\omega^n}\times B^k$. 
		We now use Fubini's theorem and convolution estimates to get
		\begin{equation}
			\label{eq:dense1}
			\begin{aligned}
				\intl_{B^k}dt \intl_{\R^n}\big|\Dsnhalf&\big((\eta_{\nu_\eps}u)*_k\rho_h\big)|^2dx
				=
				\intl_{\R^n}|\xi|^{2s}d\xi\intl_{B^k} \big|\cF_n[(\eta_{\nu_\eps}u)*_k\rho_h]\big|^2~\!dt\\ 
				&\le  
				\intl_{B^k}dt \intl_{\R^{n}}|\xi|^{2s}\big|\cF_n[\eta_{\nu_\eps}u]\big|^2d\xi
				= \intl_{B^k}dt\intl_{\R^{n}}|\Dsnhalf(\eta_{\nu_\eps}u)|^2dx\,,
			\end{aligned}
		\end{equation}
		from which we infer that the sequence $((\eta_{\nu_\eps}u)*_k\rho_h)_h$ is bounded in $\Hdiesis(\omega^n\times B^k)$. Since in addition $(\eta_{\nu_\eps}u)*_k\rho_h \to \eta_{\nu_\eps}u$ in $L^2(\R^n\times B^k)$ as $h\to \infty$, from (\ref{eq:dense1}) it indeed follows that
		$$
		(\eta_{\nu_\eps}u)*_k\rho_h \to \eta_{\nu_\eps}u\qquad \text{in $\Hdiesis(\omega^n\times B^k)$\,,}
		$$
		use the uniform convexity of $\Hdiesis(\omega^n\times B^k)$. Thus we can find a large integer $h_\eps$ such that
		\begin{equation}
			\label{eq:B}
			\intl_{B^k}dt \intl_{\R^n}\big|\Dsnhalf\big((\eta_{\nu_\eps}u)*_k\rho_{h_\eps}- \eta_{\nu_\eps}u)\big|^2dx<\eps\,.
		\end{equation}
		Next, we estimate
		$$
		\begin{aligned}
			&\intl_{\omega^n}dx \intl_{\R^k}\big|\DshalfUno\big((\eta_{\nu_\eps}u)*_k\rho_{h_\eps}\big)\big|^2dt
			=\intl_{\omega^n}dx\intl_{\R^k}|\tau|^{2s}\big|\cF_k[(\eta_{\nu_\eps}u)*_k\rho_{h_\eps}]\big|^2d\tau
			\\
			&\quad=(2\pi)^\frac k2\! \intl_{\omega^n}dx\intl_{\R^k}\big|\cF_k[\eta_{\nu_\eps}u]\big|^2\big||\tau|^s\cF_k[\rho_{h_\eps}]\big|^2d\tau
			\le (2\pi)^\frac k2\big\||\cdot|^s\cF_k[\rho_{h_\eps}]\big\|_{L^\infty(\R^k)}^2\!\iintl_{\omega^n\times B^k}\!|u|^2 dxdt\,.
		\end{aligned}
		$$
		The constant appearing in the last term is finite, as $\rho_{h_\eps}\in \mathcal C^\infty_c(\R^k)$. 
		Therefore, thanks to (\ref{eq:B}), \eqref{eq:dense1}, and to the first inequality in (\ref{eq:second_bis}), see Lemma \ref{L:poincare}, we can conclude that 
		$(\eta_{\nu_\eps}u)*_k\rho_{h_\eps}\in \tHs(\omega^n\times B^k)$. Thus there exists $u_\eps\in \mathcal C^\infty_c(\omega^n\times B^k)$ such that
		$$
		\iintl_{\R^{n+k}}\big|\Dsnphalf\big((\eta_{\nu_\eps}u)*_k\rho_{h_\eps}- u_\eps)|^2dxdt<\eps\,,
		$$
		which implies 
		\begin{equation}
			\label{eq:C}
			\intl_{B^k}dt \intl_{\R^n}\big|\Dsnhalf\big((\eta_{\nu_\eps}u)*_k\rho_{h_\eps}- u_\eps)|^2dx<\eps\,,
		\end{equation}
		by (\ref{eq:first}) in Lemma \ref{L:poincare}. To conclude the proof, write
		$$
		u-u_\eps=\big(u-\eta_{\nu_\eps}u\big)+\big(\eta_{\nu_\eps}u-(\eta_{\nu_\eps}u)*_k\rho_{h_\eps}\big)+\big((\eta_{\nu_\eps}u)*_k\rho_{h_\eps}-u_\eps\big),
		$$
		then use the triangle inequality together with \eqref{eq:A}, \eqref{eq:B}, and \eqref{eq:C}.
	\QED
	
\begin{Remark}
Let $u=u(x)\in \tHs(\omega^n)$. We can regard at $u$ as a function on $\R^n\times B^k$ which is constant in the $t$-variable,
so that $u\in L^2(\R^n\times B^k)$. 
	With this agreement and thanks to Lemma \ref{embeddings}, we have that the space
	$\tHs(\omega^{n})$ is isometrically embedded into $\Hdiesis(\omega^n\times B^k)$.
	\end{Remark}

\section{$ G$-convergence and $\Gamma$-convergence}
\label{S:GGamma}

	\paragraph{Proof of Theorem \ref{Pb_lin_2}.} 
We start with a preliminary remark. 
Since $\mathcal F_{n+k}=\mathcal F_k\circ\mathcal F_n$, Plancharel's theorem gives
$$
\begin{aligned}
\intl_{B^k}dt\irn |\xi|^{2s}|\mathcal F_n[v_\ell]|^2d\xi&=\intl_{\R^k}d\tau\irn |\xi|^{2s}|\mathcal F_{n+k}[v_\ell]|^2d\xi
\\ &
\le \intl_{\R^k}d\tau\irn (|\xi|^{2}+\ell^{-2}|\tau|^2)^s|\mathcal F_{n+k}[v_\ell]|^2d\xi\,.
\end{aligned}
$$
Thus, from (\ref{eq:Lsell0}) and (\ref{eq:Dirichlet_ell}) we get
\begin{equation}
\label{eq:rk}
\intl_{B^k}dt\irn |\Dsnhalf v_\ell|^2~\!dx\le
			\langle\mathcal L^s_\ell v_\ell,v_\ell\rangle= \iint\limits_{\omega^n\times B^k} g_\ell v_\ell~\! dxdt~\!,
\end{equation}
and therefore, using also (\ref{eq:first}), we can estimate
		$$
		\begin{aligned}
			\intl_{B^k}dt\irn |\Dsnhalf v_\ell|^2~\!dx&
			\le\Big(\iint\limits_{\omega^n\times B^k}|g_\ell|^2 dxdt\Big)^\frac12
			\Big(\intl_{B^{k}_1}dt \intl_{\omega^n} |v_\ell|^2~\!dx\Big)^\frac12
			\\&
			\le \frac{1}{\sqrt{\lambda^s_1(\omega^{n})}}\Big(\iint\limits_{\omega^n\times B^k}|g_\ell|^2 dxdt\Big)^\frac12
			\Big(\intl_{B^{k}_1}dt \intl_{\R^n} |\Dsnhalf v_\ell|^2~\!dxdt\Big)^\frac12.
		\end{aligned}
		$$
		This implies that the sequence $(v_\ell)_\ell$ is bounded in $\Hdiesis(\omega^n\times B^k)$. Thus we can find $$\hat v\in \Hdiesis(\omega^n\times B^k)$$ and a subsequence $\ell_h\to\infty$
		such that
		$$
		v_{\ell_h}\to \hat v\qquad\text{weakly in }\,\Hdiesis(\omega^n\times B^k)\ \,\mbox{and in }\,L^2(\omega^n\times B^k)\,.
		$$ 
		Now we show that $\hat v$ coincides with the unique solution to (\ref{eq:Dirichlet_omega}). 
		We take any test function $\f\in \mathcal C^\infty_c(\omega^n\times B^k)$ and set
		$$
		\Phi_h(\xi,\tau):=(|\xi|^2+{\ell_h}^{\!-2}|\tau|^2)^s\cF_{n+k}[\f](\xi,\tau)\,.
		$$
		Then $\Phi_h(\xi,\tau) \to |\xi|^{2s}\mathcal F_{n+k}[\f](\xi,\tau)$ almost everywhere on $\R^{n+k}$ and 
		$$
		|\Phi_h(\xi,\tau)|\le \big(|\xi|^2+|\tau|^2\big)^{\!s}|\cF_{n+k}[\f](\xi,\tau)|\in L^2(\R^{n+k})\,,
		$$
		because
		$$
		\iint\limits_{\R^{n+k}}(|\xi|^2+|\tau|^2)^{2s}|\cF_{n+k}[\f]|^2d\xi d\tau=\int\limits_{\R^{n+k}}|\zeta|^{4s}|\cF_{n+k}[\f]|^2~\!d\zeta
		=\int\limits_{\R^{n+k}}|(-\Delta_{n+k})^s\f|^2~\!dz\,.
		$$
		Thus Lebesgue's theorem gives
		$$
		\Phi_h(\xi,\tau)\to|\xi|^{2s}\cF_{n+k}[\f]\quad \text{in}\ \,L^2(\R^{n+k})\,.
		$$
		Since $v_{\ell_h}\to \hat v$ weakly in $L^2(\omega^n\times B^k)$, then $\cF_{n+k}[v_{\ell_h}]\to \cF_{n+k}[\hat v]$ weakly in $L^2(\R^{n+k})$. 
		Therefore, using also Plancharel's theorem, we have
		$$
		\begin{aligned}
			\langle \mathcal L^s_{\ell_h}v_{\ell_h},\f\rangle &=\iintl_{\R^{n+k}}\cF_{n+k}[v_{\ell_h}]~\!\overline{\Phi_h}~d\xi d\tau=
			\iintl_{\R^{n+k}}\!|\xi|^{2s}\cF_{n+k}[\hat v]~\!\overline{\cF_{n+k}[\f]}\,d\xi d\tau +o(1)\\
			&=\iintl_{\R^{n+k}}\!|\xi|^{2s}\cF_{n}[\hat v]~\!\overline{\cF_n[\f]}\,d\xi d\tau +o(1)=\intl_{B^k}dt\irn(\Dsnhalf \hat v)(\Dsnhalf \f)~\!dx+o(1)\,.
		\end{aligned}
		$$
		On the other hand, since $v_{\ell_h}$ solves \eqref{eq:Dirichlet_ell} and $g_{\ell_h}\to g$ in $L^2(\omega^n\times B^k)$, we have that
		$$
		\langle \mathcal L^s_{\ell_h}v_{\ell_h},\f\rangle=\iintl_{\omega^n\times B^k}g_{\ell_h} \f~\!dxdt=\iintl_{\omega^n\times B^k}g \f~\!dxdt+o(1)\,.
		$$
		In conclusion,  we have  that 
		$$
		\intl_{B^k}dt\irn (\Dsnhalf \hat v)(\Dsnhalf \f)~\!dx= \iintl_{\omega^n\times B^k}g\f~\!dxdt\quad \text{for any $\f\in \mathcal C^\infty_c(\omega^n\times B^k)$\,,}
		$$
		that is, $\hat v=v$, where $v$ is the unique solution to \eqref{eq:Dirichlet_omega}. Now we use again (\ref{eq:rk}) to get
		$$
		\begin{aligned}
			\intl_{B^k}dt\irn |\Dsnhalf v_{\ell_h}|^2~\!dx&
			\le \iint\limits_{\omega^n\times B^k}g_{\ell_h} v_{\ell_h}~\!dxdt
			\\&
			=\iint\limits_{\omega^n\times B^k}gv~\!dxdt+o(1)=\intl_{B^k}dt\irn |\Dsnhalf v|^2~\!dx
		\end{aligned}
		$$
		by \eqref{eq:Dirichlet_omega}. 
		This is sufficient to infer that $v_{\ell_h}\to v$ in the uniformly convex space $\Hdiesis(\omega^n\times B^k)$.
		In fact, due to the uniqueness of the limit of all (weakly) convergent subsequences $v_{\ell_h}$, we can conclude that
		the full sequence $v_\ell$ converges to $v$ in  $\Hdiesis(\omega^n\times B^k)$, as $\ell\to\infty$.
\QED
	
\paragraph{Proof of Theorem \ref{Pb_lin_new}}
		Let $f_\ell$, $u_\ell$ and $f_\infty=f_\infty(x)$ be as in the statement. We put
		$$
		\invbreve f_\ell(x,t):=f_\ell(x,\ell t)\,,\qquad\invbreve u_\ell(x,t):=u(x,\ell t)\,.
		$$
		It is easy to show that $\invbreve f_\ell\in L^2(\omega^n\times B^k)$ and that $\invbreve u_\ell\in \tHs(\omega^n\times B^k)$ is a  solution to
		$$
		\mathcal L^s_\ell \invbreve u_\ell= \invbreve f_\ell\qquad \text{in $\tHs(\omega^n\times B^k)'$.}
		$$
		Since $f_\ell\to f_\infty$ in $L^2-$mean, then $\invbreve f_\ell\to f_\infty$ in $L^2(\omega^n\times B^k)$. Therefore, Theorem \ref{Pb_lin_2} implies that $\invbreve u_\ell\to v$ in $\Hdiesis(\omega^n\times B^k)$, where
		$v\in \Hdiesis(\omega^n\times B^k)$ is the unique  solution to 
		\begin{equation}\label{eq:almost}
			\Dsn v= f_\infty\qquad \text{in $\Hdiesis(\omega^n\times B^k)'$}
		\end{equation}
		(recall that $f_\infty$ does not depend on $t$).
		
		To prove that $v(x,t)=u_\infty(x)$ on $\omega^n\times B^k$, let $\f\in \tHs(\omega^n)$ and $\psi\in L^2(B^k)$. 
		Testing (\ref{eq:almost}) with $\f\psi\in\Hdiesis(\omega^n\times B^k)$, we get 
		$$
		\intl_{B^k}\psi(t)\,dt\irn(\Dsnhalf v)(\Dsnhalf\f)\,dx=\intl_{B^k}\psi(t)\,dt\irn f_\infty\f\,dx\,.
		$$
		Since the choice of  $\psi\in L^2(B^k)$ was arbitrary, we see that for a.e. $t\in B^k$ it holds that
		$$
		\irn(\Dsnhalf v)(\Dsnhalf \f)\,dx=\int\limits_{\omega^n} f_\infty \f\,dx\quad\ \text{for any }\,\f\in \tHs(\omega^n)\,.
		$$
		Hence, for a.e. $t\in B^k$ the function $v(\cdot,t)$ 
		coincides with the unique solution $u_\infty\in\tHs(\omega^n)$ to \eqref{eq:Dirichlet_omega_Thm}. 
		
		Finally, recall that  $\invbreve u_\ell\to u_\infty$ in $\Hdiesis(\omega^n\times B^k)$ means
		$$
			\intl_{B^k}dt\irn\big|\Dsnhalf \invbreve u_{\ell}-\Dsnhalf v_\infty\big|^2 dx \to 0\,.
		$$
		Rescaling back, one can plainly prove that $\Dsnhalf u_{\ell}\to \Dsnhalf u_\infty$ in $L^2-$mean. Theorem \ref{Pb_lin_new} is completely proved. 
\QED

\begin{Remark}\label{R:AFM1}
Let $u_\ell\in \tHs(\omega^n\times\ell B^k)$, $u_\infty\in \tHs(\omega^n)$ be as in the statement of Theorem \ref{Pb_lin_new}. 
Then
$$
\rho_\ell u_\ell:=\fint\limits_{\ell B^k} u_\ell(\cdot,t)~\!dt\to u_\infty \qquad \text{in $\tHs(\omega^n)$\,,}
$$
compare with \cite[Theorem 1.2]{AFM}, where this conclusion is proved in case $s\in(0,1)$ and $B^k$ convex. In fact, notice that 
$$
	\begin{aligned}
	\irn\big|\Dsnhalf&(\rho_\ell u_\ell- u_\infty)|^2~\!dx=
	\irn|\xi|^{2s}\big|\cF_n[\rho_\ell u_\ell- u_\infty]\big|^2d\xi\\
	&\le\irn|\xi|^{2s}d\xi \fint\limits_{\ell B^k}|\cF_n[u_\ell- u_\infty]|^2~\!dt=\fint\limits_{\ell B^k}dt\irn\big|(-\Delta_n)^{\frac s2}(u_\ell- u_\infty)|^2dx
	\to 0\,,
	\end{aligned}
	$$
by the H\"older inequality, and thanks to (\ref{eq:tesi}). 

\end{Remark}

\appendix

\section{\hskip-0.56cmppendix}

Here we point out a few improvements of our results, as well as some remarks connecting our asymptotic analysis with the available literature.

\subsection{The infinite tube}
	It holds that 
	\begin{equation}
		\label{eq:equal}
		{\lambda^s_1}(\omega^n\times \R^k):=\inf_{u\in \tHs(\omega^n\times \R^k)\atop u\neq 0}\frac{\displaystyle\iintl_{\R^{n+k}}|\Dsnhalf u|^2~\!dxdt}
		{\displaystyle\iintl_{\omega^n\times \R^k}| u|^2~\!dxdt}
		={\lambda^s_1}(\omega^n)\,.
	\end{equation}
	In fact, since $\mathcal C^\infty_c(\omega^n\times \R^k)$ is dense in $H^s_0(\omega^n\times \R^k)$, then
	\eqref{eq:first} gives ${\lambda^s_1}(\omega^n\times \R^k)\ge {\lambda^s_1}(\omega^n)$. 
	On the other hand, for any $\ell>0$ we have 
	${\lambda^s_1}(\omega^n\times \R^k)\le {\lambda^s_1}(\omega^n\times\ell B^k)$ because
	$\tHs(\omega^n\times\ell B^k)\hookrightarrow \tHs(\omega^n\times \R^k)$. Thus (\ref{eq:equal}) 
	follows by taking the limit as $\ell\to \infty$, thanks to Theorems \ref{T:eigenvalue1} and \ref{T:eigenvalue2}.
	
	The equality (\ref{eq:equal}) was proved in \cite{AFM} in case $s\in(0,1)$ via the maximum principle; see also \cite{BN} where $s\in(0,1)$, $k=1$
	are assumed and a different argument is used.

\subsection{Improvements of Theorem \ref{T:eigenvalue2}}
Throughout this subsection, we assume that 
$s>1$.

\subsubsection{Estimates on higher order eigenvalues.}
\label{SSS:lambda_m}
Assume that $\lambda^s_1(\omega^n)$ is simple, which happens for instance in case $n=1$ and $\omega^n$ is an interval, see \cite{AJS,DG}. Since $s>1$, by (\ref{eq:second_bis}) we have that
$$
\widetilde\E^s_{\ell}(u)\le \iint\limits_{\R^{n+k}} |\Dsnphalf u|^2~\! dxdt  \quad\ \, \text{for any $u\in\tHs(\omega^n\times\ell B^k)$\,.}
$$
Therefore, using the Courant–Fischer–Weyl variational principle and the same argument as in the proof of 
Theorem \ref{T:eigenvalue1} we get that
$\lambda^s_m(\omega^n\times\ell B^k)$ is not smaller than $\tlambda(\omega^n\times\ell B^k)$, which is the 
$m$-th eigenvalue of the operator $\Dsn\oplus\DsUno$ on $\tHs(\omega^n\times\ell B^k)$. 
Since $\lambda_1^s(\omega^n)$ is simple, then $\tlambda(\omega^n\times\ell B^k)$ can be computed
for $\ell$ large enough, leading to the inequality 
	\begin{equation*}
		\lambda^s_1(\omega^n)+\frac{\lambda^s_m(B^k)}{\ell^{2s}}<\lambda_m^s(\omega^n\times\ell B^k)\,,
	\end{equation*}
	which holds provided that $\ell>0$ is large enough, depending on $m$.

\subsubsection{Estimates for ${s>>1}$ large.}
In view of (\ref{eq:s=1}), one can not expects to obtain significant improvements of (\ref{eq:weak}), which include powers $s\searrow 1$. 
On the other hand, the estimates in \eqref{eq:weak} can be progressively refined as the integer part of $s$ increases. 
This can be done via elementary numerical inequalities. For the sake of readability, we restrict our attention to the case
$s\ge 2$.

\begin{Lemma}\label{L:numbers}
	Let $s\ge 2$. For any $\zeta>0$, the following inequalities hold:
	\begin{equation}
	\label{eq:elementary2}
		1+\zeta^s+s\zeta\le (1+\zeta)^s<1+(s-1)c_{s-1}\zeta^s+s\zeta+\frac{s(s-1)}2c_{s-1}\zeta^2,
	\end{equation}
	where $c_{s-1}=\max\{1, 2^{s-3}\}$ (compare with Theorem \ref{T:eigenvalue2}). 
 \end{Lemma}
 
 \proof
 We can assume that $s>2$.
 Consider the functions $f, F:(0,\infty)\to \R$ given by
 $$
f(\zeta)=\zeta^{-1}((1+\zeta)^s-1-\zeta^{s})~,\quad  F(\zeta):=\zeta^2f'(\zeta)=(1+\zeta)^{s-1}(s\zeta-\zeta-1)+1-(s-1)\zeta^s\,.
 $$ 
 We compute $F'(\zeta)=s(s-1)\zeta\left((1+\zeta)^{s-2}-\zeta^{s-2}\right)$, so that $F$ is increasing. 
 Thus $f$ is increasing on $(0,\infty)$ as well, because $F(0)=0$. Since (in the limit) $f(0)=s$, we infer that 
 $1+\zeta^s+s\zeta<(1+\zeta)^s$.

\medskip
Next, we use (\ref{eq:elementary}) with $s-1$ instead of $s$ to estimate
$$
(1+\zeta)^s=1+s\int_0^\zeta (1+y)^{s-1}dy< 1+s\int_0^\zeta\big[1+(s-1)c_{s-1}\big(y+y^{s-1}\big)]dy\,,
$$
which readily gives the right hand side inequality in (\ref{eq:elementary2}).
\QED

	\begin{Theorem}\label{T:eigenvalue3}
		Let $s\ge 2$. For any $\ell>0$ it holds that
		$$
		\begin{aligned}
		\lambda^s_1(\omega^n)+ \frac{\lambda^s_1(B^k)}{\ell^{2s}} +{s}\frac{\lambda_1^{s-1}(\omega^n)\lambda_1^1(B^k)}{\ell^2}
		&<\lambda^s_1(\omega^n\times\ell B^k)\\
		&<\lambda_1^s(\omega^n)+(s-1)c_{s-1}\frac{\lambda_1^s(B^k)}{\ell^{2s}}+s \frac{\gamma_{s}}{\ell^2}+\frac{s(s-1)}2c_{s-1}\frac{\delta_s}{\ell^4}\,,
		\end{aligned}
		$$
		where $\gamma_s$ is defined in Theorem \ref{T:eigenvalue2}, 
$$
\begin{aligned}
\delta_s&=\irn|\left(-\Delta_{n}\right)^{\frac{s-2}{2}}\!\f_1|^2\,dx\intl_{B^k}|\Delta\psi_1|^2 \,dt
\le \lambda^s_1(\omega^n)^\frac{s-2}s\lambda^s_1(B^k)^\frac2s\,,
\end{aligned}
$$
and $\f_1\in H^s_0(\omega^n)$, $\psi_1\in H^s_0(B^k)$ are as in the proof of Theorem \ref{T:eigenvalue2}.
	\end{Theorem}
	
	\proof 
	Lemma \ref{L:numbers} implies that
	$$
	(|\xi|^2+|\tau|^2)^s\ge |\xi|^{2s}+|\tau|^{2s}+{s}|\xi|^{2(s-1)}|\tau|^2\qquad \text{for a.e. $\xi\in \R^n,\,\tau\in\R^k$.}
	$$
	Therefore, letting $\phi\in H^{s}_0(\omega^n\times\ell B^k)$ be an eigenfunction for $\lambda_1^s(\omega^n\times\ell B^k)$ with unitary $L^2$-norm, we have that
	$$
	\begin{aligned}
		\lambda_1^s&(\omega^n\times\ell B^k)= \iintl_{\R^{n+k}}|\Dsnphalf \phi|^2d\xi d\tau=\iint\limits_{\R^{n+k}}(|\xi|^2+|\tau|^2)^s|\cF_{n+k}[{\phi}]|^2~\!d\xi d\tau\\
		&\ge \iintl_{\R^{n+k}}|\xi|^{2s}|\cF_{n}[{\phi}]|^2d\xi d\tau
		+\iintl_{\R^{n+k}}|\tau|^{2s}|\cF_{k}[{\phi}]|^2d\xi d\tau+s\intl_{\R^n}|\xi|^{2(s-1)}d\xi\intl_{\R^k} |\tau|^2|\cF_{k}[\mathcal F_n[{\phi}]]|^2d\tau
	\end{aligned}
	$$
	Easily, we can estimate
	$$
	\iintl_{\R^{n+k}}|\xi|^{2s}|\cF_{n}[{\phi}]|^2d\xi d\tau +
	\iintl_{\R^{n+k}}|\tau|^{2s}|\cF_{k}[{\phi}]|^2d\xi d\tau\ge \lambda_1^s(\omega^n)+\ell^{-2s}\lambda_1^s(B^k)\,.
	$$
	In fact, the strict inequality holds. Otherwise, $\phi(\cdot, t)$ would achieve  $\lambda_1^s(\omega^n)$ for a.e. $t\in \ell B^k$ 
	and  
	$\phi(x,\cdot)$ would achieve $\lambda_1^s(\ell B^k)$ for a.e. $x\in\omega^n$, which is impossible.

	Moreover 
	$$
	\begin{aligned}
\intl_{\R^n}|\xi|^{2(s-1)}d\xi&\intl_{\R^k} |\tau|^2|\cF_{k}[\mathcal F_n[{\phi}]]|^2d\tau=
		\intl_{\R^n}|\xi|^{2(s-1)}d\xi\intl_{\ell B^k} |\nabla_{\!t}(\mathcal F_n[{\phi}])|^2dt\\
		&\ge \ell^{-2}\lambda_1^1(B^k)\intl_{\ell B^k} dt  \intl_{\R^n}|\xi|^{2(s-1)} |\mathcal F_n[{\phi}]|^2d\xi
		\\&
		=\ell^{-2}\lambda_1^1(B^k)\intl_{\ell B^k} dt  
		\intl_{\R^n}|\left(-\Delta_{n}\right)^{\frac{s-1}{2}} {\phi}|^2 dx
		\ge \ell^{-2}\lambda_1^1(B^k)\lambda_1^{s-1}(\omega^n)\,,
	\end{aligned}
	$$
	which concludes the proof of the lower bound on $\lambda_1^s(\omega^n\times\ell B^k)$.
	
	\medskip
	
	Next, notice that Lemma \ref{L:numbers} implies the estimate
	$$
	(|\xi|^2+\ell^{-2}|\tau|^2)^s< |\xi|^{2s}+(s-1)c_{s-1}\ell^{-2s}|\tau|^{2s}+s\ell^{-2}|\xi|^{2(s-1)}|\tau|^2+\frac{s(s-1)}2c_{s-1}\ell^{-4}|\xi|^{2(s-2)}|\tau|^4,
	$$
	which holds for a.e. $\xi\in \R^n,\tau\in\R^k$.
Therefore, arguing as in the proof of Theorem \ref{T:eigenvalue2} we can estimate
\begin{multline*}
	\lambda_1^s(\omega^n\times\ell B^k)<\lambda_1^s(\omega^n)+(s-1)c_{s-1}\ell^{-2s}\lambda_1^s(B^k)\\
	+s\ell^{-2}\irn|\xi|^{2(s-1)}|\cF_n[\f_1]|^2d\xi\intl_{\R^k}|\tau|^2|\cF_k[\psi_1]|^2d\tau\\
	+\frac{s(s-1)}2c_{s-1} \ell^{-4}\irn|\xi|^{2(s-2)}|\cF_n[\f_1]|^2d\xi\intl_{\R^k}|\tau|^4|\cF_k[\psi_1]|^2d\tau.
\end{multline*}
The upper bound on $\lambda_1^s(\omega^n\times\ell B^k)$ follows, thanks to the expression for $\gamma_s$ given by (\ref{eq:gamma}) and \eqref{eq:gamma0},
and since
$$
\irn|\xi|^{2(s-2)}|\cF_n[\f_1]|^2\,d\xi\intl_{\R^k}|\tau|^4|\cF_k[\psi_1]|^2\,d\tau=
\irn|\left(-\Delta_n\right)^{\frac{s-2}{2}}\!\f_1|^2\,dx\intl_{B^k}|\Delta \psi_1|^2 \,dt\,.
$$
The estimate on $\delta_s$ is readily proved via the H\"older inequality. 
\QED

\subsubsection{A result by M.P. Owen \cite{O}.}
\label{SSS:Owen}

From now on we set
$$
I=(-1,1)\,.
$$

Let $s=2$, $n=k=1$ and take $\omega^n=B^k=I$. To avoid ambiguity, we use the notation $\lambda^{\!(2)}_1$ instead of $\lambda^2_1$. We deal with the eigenvalue problem
		\begin{equation}
			\label{eq:Owen}
			\begin{cases}
				\Delta^2 u=\lambda^{\!(2)}_1\!(I\times \ell I)\,u&\text{in $I\times \ell I$\,,}\\
				u=0\,,~~\partial_n u=0&\text{on $\partial (I\times \ell I)$\,.}
			\end{cases}
		\end{equation}
		Numerical computations give $\lambda^{\!(2)}_1\!(I)\approx 31.285$. In case $s=2$, Theorem \ref{T:eigenvalue3} gives
		$$
		\lambda^{\!(2)}_1(I)+2\frac{\pi^4}{16\ell^2}+\frac{\lambda^{\!(2)}_1(I)}{\ell^4}
		<\lambda^{\!(2)}_1\!(I\times \ell I) <
		\lambda^{\!(2)}_1(I)+2\frac{\|\partial_x\f_1\|_2^4}{\ell^2}+2\frac{\lambda^{\!(2)}_1(I)}{\ell^4}\,,
		$$
			since
			${\pi^2}/{4}$ is the principal value for $-\partial^2_{x}$ on $H^1_0(I)$. It holds that
			$$
\frac{\pi^4}{16}<\|\partial_x\f_1\|_2^4<{\lambda^{\!(2)}_1\!(I)}\,.
			$$		
		
		In \cite{O}, M.P. Owen obtained the following asymptotic estimate:
		$$
		\lambda^{\!(2)}_1\!(I\times \ell I) = \lambda^{\!(2)}_1\!(I)+2 D\frac{\sqrt{\lambda^{\!(2)}_1\!(I)}}{\ell^{2}}+O(\ell^{-3})\,,
		$$
		where $D\approx 5.4270$ is implicitly defined via a transcendental equation.

\subsection{Thin domains}

As a consequence of Theorems \ref{T:eigenvalue1} and \ref{T:eigenvalue2}, 
we obtain the following corollary for  domains of the form
$$
\eps I\times A\subset \R\times\R^k\,, \qquad \eps\to 0\,.
$$

\begin{Corollary}
	\label{T:Braides1}
	Let $A\subset\R^k$ be a bounded open set in $\R^k$ and let $I$ be the interval $(-1,1)$.
	\begin{itemize}
		\item[$i)$] If $s\in(0,1)$, then for any integer $m\ge 1$ it holds that   
		$$\lambda^s_1(I)<\eps^{2s}\lambda^s_m(\eps I\times A) < \lambda^s_1(I)+\eps^{2s}{\lambda^s_m(A)}\,,$$
		provided that $\eps>0$ is small enough;
		\item[$ii)$]	 If $s>1$, then
		$$
		\lambda^s_1(I)+\eps^{2s}\lambda^s_1(A)< \eps^{2s}\lambda^s_1(\eps I\times A) 
		< \lambda^s_1(I)+\gamma_s~\! \eps^2+ c_s\eps^{2s}\lambda^s_1(A)
		$$
		for any $\eps>0$, where the constants $c_s, \gamma_s$ are as in Theorem \ref{T:eigenvalue2}.
	\end{itemize}
\end{Corollary}

\proof
It is not restrictive to assume that $0\in A$.
Notice that the transform
\begin{equation}\label{transform}
	u(z)\mapsto \invbreve u_\eps(z):=\eps^{-\frac{1+k-2s}{2}}u\Big(\frac{z}{\eps}\Big)
\end{equation}
is an isometry $\tHs(\eps I\times A)\to \tHs\big(I\times \tfrac1\eps A\big)$, while
$$
\intl_{-\eps}^\eps dx\intl_A|u|^2\,dt= \eps^{-2s}\intl_{-1}^1 dx\intl_A|\invbreve u_\eps|^2\,dt\,.
$$
Thus $\lambda^s_m(\eps I\times A)=\eps^{-2s}\lambda^s_m(I\times \tfrac1\eps A)$.
The Corollary follows by 
taking $n=1$ and $\omega^n=I$ in Theorems  \ref{T:eigenvalue1} and \ref{T:eigenvalue2}, and by replacing $\ell$ with $1/\varepsilon$.
\QED

\begin{Remark}
	For any $s>1$ the eigenvalue  $\lambda_1^s(I)$ is simple, see \cite{AJS,DG}. Hence the remarks in Subsection \ref{SSS:lambda_m}, combined with the transform \eqref{transform}, imply that  for any fixed integer $m\geq2$ the estimate
	$$
	\lambda^s_1(I)+\eps^{2s}\lambda^s_m(A)<\eps^{2s}\lambda^s_m(\eps I\times A)
	$$
	holds, provided that $\eps>0$ is small enough.
\end{Remark}

\subsubsection{A result by F.L. Bakharev and S.A. Nazarov \cite{BSAN}.}
\label{SSS:BSAN}

Here use the same notation as  in Subsection \ref{SSS:Owen}. 
The eigenvalue problem
	$$
	\begin{cases}
		\Delta^2 u=\lambda^{\!(2)}_m\!(\eps I\times I)\,u&\text{in $\eps I\times I=(-\eps,\eps)\times (-1,1)\,,$}\\
		u=0\,,~~\partial_n u=0&\text{on $\partial (\eps I\times I)$}
	\end{cases}
	$$
	is equivalent to (\ref{eq:Owen}). Therefore, the results in \cite{O} provide the asymptotic behaviour of the principal eigenvalue $\lambda^{\!(2)}_1\!(\eps I\times I)$.
	
	In \cite{BSAN}, F.L. Bakharev and S.A. Nazarov were able to provide asymptotic estimates on $\lambda^{\!(2)}_m(\eps I\times I)$
	for any integer $m\ge 1$. More precisely,
	they proved that there exists a constant $C_m$ not depending on $\eps$, such that
	$$
	2\frac{\pi^2}{4}M\eps^2m^2-C_m\eps^\frac52< \eps^4\lambda^{\!(2)}_m(\eps I\times I)-\lambda^{\!(2)}_1\!(I)
	<2 \frac{\pi^2}{4}M\eps^2m^2+C_m\eps^\frac52
	$$
	for $\eps>0$ small enough, depending on $m$. Here $M$ is an explicit constant depending on the first eigenfunction of the Dirichlet biharmonic problem on the segment.

\paragraph{Acknowledgments.}
G. Romani is partially supported by the INdAM-GNAMPA 2026 project \textit{Structural degeneracy and criticality in (sub)elliptic PDEs} (E53C25002010001).

\end{document}